\documentclass[11pt,a4paper]{article}
\usepackage[dvipsnames,svgnames,table]{xcolor}
\usepackage{sectsty}
\usepackage{multirow}
\usepackage{authblk}
\usepackage{relsize}
\usepackage[T1]{fontenc} % Use 8-bit encoding that has 256 glyphs
\usepackage{fourier} % Use the Adobe Utopia font for the document - comment this line to return to the LaTeX default
\usepackage[english]{babel} % English language/hyphenation
\usepackage{amsmath,amsfonts,amsthm, amssymb} % Math packages
\usepackage{graphicx}
\usepackage{booktabs}
\usepackage[font=small]{subcaption}
\usepackage[font=small]{caption}
\usepackage{svg}
\usepackage{wrapfig}
\usepackage{bm}
\usepackage{soul} % Underline
\usepackage{lipsum} % Used for inserting dummy 'Lorem ipsum' text into the template
\usepackage{sectsty} % Allows customizing section commands
\usepackage{fancyhdr} % Custom headers and footers
\usepackage[T1]{fontenc} % Use 8-bit encoding that has 256 glyphs
\usepackage[english]{babel} % English language/hyphenation
\usepackage{framed}
\usepackage{bm}
\usepackage{fancyhdr}
\usepackage{etoolbox}
\usepackage{titlesec}
\usepackage{titletoc}
\usepackage{float}
\usepackage{xcolor}
\theoremstyle{plain}% default
\newtheorem{theorem}{Theorem}[section]
\newtheorem{lemma}{Lemma}[section]
\newtheorem{proposition}{Proposition}[section]
\newtheorem{corollary}{Corollary}[section]

\theoremstyle{definition}

\theoremstyle{remark}

\newtheorem{note}{Note}

\renewcommand{\v}[1]{\boldsymbol{\mathbf{#1}}} % Black both Greek and Latin letters.
\newcommand{\overbar}[1]{\mkern 1.5mu\overline{\mkern-1.5mu#1\mkern-1.5mu}\mkern 1.5mu}
\providecommand{\keywords}[1]
{
  \small	
  \textbf{\textit{Keywords: }} #1
}

\let\div\undefined
\DeclareMathOperator{\div}{div}

\numberwithin{equation}{section} % Number equations within sections (i.e. 1.1, 1.2, 2.1, 2.2 instead of 1, 2, 3, 4)
\numberwithin{figure}{section} % Number figures within sections (i.e. 1.1, 1.2, 2.1, 2.2 instead of 1, 2, 3, 4)
\numberwithin{table}{section} % Number tables within sections (i.e. 1.1, 1.2, 2.1, 2.2 instead of 1, 2, 3, 4)

\definecolor{grisf}{rgb}{.47,.47,.47} % barre de droite gris fonce
\newcommand{\colorc}{\color{SkyBlue}}
\newcommand{\colorb}{\color{Cerulean}}
\newcommand{\colora}{\color{Blue}}

\titleformat{\chapter}[display]
  {\normalfont\sffamily\bfseries\huge\colora\centering}{\thechapter}{1ex}
  {{\titlerule[1pt]}\vspace{1.3ex}}[\vspace{1ex}{{\titlerule[1pt]}}]

\titleformat{name=\chapter,numberless}[display]
  {\normalfont\sffamily\bfseries\LARGE\colora\centering}{}{1ex}
  {{\titlerule[1pt]}\vspace{1.3ex}}[\vspace{1ex}{\titlerule[1pt]}\vspace{2ex}]

\titleformat{\section}[hang]{\Large\normalfont\sffamily\bfseries\colorb}{{\thesection\, }}{0 em}
  {}[{\titlerule[1pt]}\vspace{1ex}]

\titleformat{\subsection}[hang]{\Large\normalfont\sffamily\bfseries\colorc}{{\thesubsection\, }}{0 em}
  {}[{\titlerule}\vspace{.7ex}]
\titleformat{\subsubsection}[hang]{\normalfont\sffamily\bfseries\large}{{\thesubsubsection\, }}{0 em}
  {}[{\color{grisf}\titlerule}\vspace{3pt}]
\titleformat{\paragraph}[runin]{\normalfont\sffamily\bfseries\colorc}{}{0 em}
  {\indent}

\makeatletter
\patchcmd{\@fancyhead}{\rlap}{\color{grisf}\rlap}{}{}
\patchcmd{\headrule}{\hrule}{\color{grisf}\hrule}{}{}
\patchcmd{\@fancyfoot}{\rlap}{\color{grisf}\rlap}{}{}
\patchcmd{\footrule}{\hrule}{\color{grisf}\hrule}{}{}
\makeatother

\fancypagestyle{plain}{
  \fancyhead{}
  
  }

\title{H\"older stability of quantitative photoacoustic tomography based on partial data}

\author{
 Faouzi TRIKI\footnote{Faouzi Triki,  Laboratoire Jean Kuntzmann,  UMR CNRS 5224, 
Universit\'e  Grenoble-Alpes, 700 Avenue Centrale,
38401 Saint-Martin-d'H\`eres, France; Email: Faouzi.Triki@univ-grenoble-alpes.fr},  \; and   Qi XUE \footnote{Qi Xue,  Laboratoire Jean Kuntzmann,  UMR CNRS 5224, 
Universit\'e  Grenoble-Alpes, 700 Avenue Centrale,
38401 Saint-Martin-d'H\`eres, France; Email: Qi.Xue@univ-grenoble-alpes.fr \& xueqi.bhlt@gmail.com. \\The authors were supported by the grant ANR-17-CE40-0029 of the French National Research Agency ANR (project MultiOnde).}\\
}
\date{}

\begin{document}
\maketitle

%\linenumbers

\begin{abstract}
We consider the reconstruction of the diffusion and absorption coefficients of the diffusion equation
from the internal information of the solution obtained from the first step of the inverse photoacoustic tomography (PAT).
In practice, the internal information is only partially provided near the boundary 
due to the high absorption property of the medium and the limitation of the equipment.
Our main contribution is to prove a H\"older stability of the inverse problem in a subregion
where the internal information is reliably provided 
based on the stability estimation of a Cauchy problem satisfied by the diffusion coefficient.
The exponent of the H\"older stability converges to a positive constant independent of the subregion as 
the subregion contracts towards the boundary. 
Numerical experiments demonstrates that it is possible to locally reconstruct the diffusion and absorption coefficients
for smooth and even discontinuous media.
\end{abstract}
\pagestyle{fancy}

\section{Introduction}

PAT is a hybrid medical imaging technique which combines the high contrast of optical parameters 
with the high resolution of ultrasonic waves \cite{wang2017photoacoustic, ammari2017multi, fisher2007photoacoustic, li2009photoacoustic, kuchment2010mathematics}.
In PAT, near infra-red (NIR) photons are sent into the biological tissue which is heated up due to the absorption of the energy.
The heating then results in the expansion of the tissue which generates a pressure field.
The measurement of the pressure field on the boundary is used to reconstruct the optical properties of the tissue.

\vskip .3cm
The inverse problem of PAT can be decomposed into two steps.
The first step is to reconstruct the absorbed radiation map $H(x)$ from the measurement of ultrasonic waves on the boundary
\cite{fisher2007photoacoustic,agranovsky2007uniqueness,hristova2008reconstruction,
kunyansky2008thermoacoustic,hristova2009time,stefanov2009thermoacoustic,qian2011efficient}.
The second step is to reconstruct the diffusion coefficient $D(x)$ and the absorption coefficient $\mu(x)$
through the internal data $H(x)$ obtained in the first step
\cite{cox2006two,bal2010inverse,laufer2010quantitative,naetar2014quantitative,ren2013hybrid,ABJW, ammari2010mathematical, aspri2020asymptotic, alessandrini2017stability}.
Let us consider the Dirichlet problem in a bounded domain $\Omega\subset\mathbb{R}^n$

\begin{equation}\label{eqn:diffuse}
\left\{\begin{array}{ll}
-\div\big(D(x)\nabla u(x)\big)+\mu(x)u(x) = 0 & \text{in }\Omega,\\
u(x) = g(x) & \text{on } \partial\Omega.
\end{array}\right.
\end{equation}
We need to reconstruct $D(x)$ and $\mu(x)$ in (\ref{eqn:diffuse})
from the knowledge of the coefficients $D(x)$ and $\mu(x)$ on the boundary, the boundary condition $g(x)$,
and the internal data $H(x) = \Gamma(x)\mu(x)u(x)$, 
where $\Gamma(x)$ is the coupling coefficient quantifying the amount of ultrasound generated by photons.
It has been proved in \cite{bal2011multi} that it is impossible to reconstruct $(\Gamma,D,\mu)$
at the same time no matter how many sets of internal data are used for a fixed frequency.
In this paper, we focus on the second step, and we further assume that $\Gamma\equiv 1$.

\vskip .3cm

PAT provides in theory images of optical contrasts and ultrasound resolution. However,
in practice it has been observed in various experiments that the imaging depth, i.e., the maximal depth of the medium at which structures can be resolved at expected resolution, of PAT is still fairly limited, usually on the order of millimeters. This is mainly due to the  fact that  optical waves are  significantly attenuated by absorption and scattering. In fact the generated  optical signal decays very fast in the depth direction. This is indeed a well-known faced  issue in optical tomography~\cite{wang2017photoacoustic}.  Recently in \cite{RenTriki19}, assuming that the medium is layered, the authors  derived a stability estimate showing that  the reconstruction of the optical coefficients  is stable in the region close to the optical illumination source and deteriorate exponentially far away.  Due to the high absorption property of the tissue, the limitation of the equipment, etc., the boundary source $g(x)$ is  in practice confined near the impact zone of the near infra-red photons,   and it is impossible to illuminate the whole tissue or to take the measurement on the whole boundary (see \cite{bonnetier2019stability}
and references therein). Therefore the data $H(x)$ is only reliably provided near the boundary of measurement \cite{burgholzer2007exact,cox2007photoacoustic}.
To our best knowledge, the stability analysis based on partial data of $H(x)$ has not been addressed,
which is the motivation of this paper.

\vskip .3cm
In this paper we first derive a Cauchy problem satisfied by $\sqrt{D}$ 
whose coefficient and source term depend locally on $H(x)$ in Section \ref{sec:local_photoacoustic}.
We prove a H\"older stability of the Cauchy problem in a subregion near the boundary of measurement in Section \ref{sec:stability_cauchy},
which results in a H\"older stability estimation for the reconstruction of $D$ and $\mu$ in Section \ref{sec:stability_Dmu}.
Actually there already exists a H\"older stability estimation for the Cauchy problem inside a subregion away from the boundary;
see for example \cite{klibanov2000carleman,alessandrini2009stability,choulli2016applications,bellassoued2017carleman}.
The main drawback of the existing stability estimation is that the constant inside the upper bound tends to infinity 
while the distance between the subregion and the boundary goes to zero.
We modify their method such that the constant becomes independent of the subregion.
We also propose a choice of the exponent of the  H\"older stability estimation 
which increasingly converges to  a strictly positive constant independent of the subregion as the subregion contracts towards the boundary.
That is, we improve the existing theory to handle a subregion including the boundary of measurement and 
prove that the stability increases as the subregion becomes smaller.
The obtained stability results show that the resolution of PAT 
is better near the impact zone of the optical illumination sources, and deteriorates 
far away. %, which is in accordance with known experimental observations \cite{wang2017photoacoustic}.
At last several numerical experiments on smooth, discrete and realistic media are presented in Section \ref{sec:numeric}. Our algorithm is able to reconstruct all the inhomogeneity accurately.

%%%%%%%%%%%%%%%%%%%%%%%%%%%%%%%%%%%%%%%%%%%%%%%%%%%%%%%%%
\section{Local stability estimation}\label{sec:local_photoacoustic}
We consider the problem of reconstructing $D(x)$ and $\mu(x)$ in (\ref{eqn:diffuse}) 
from a set of internal data $H_j(x) = \mu(x)u_j(x)$, $j = 1, 2, \ldots, n+1$, where $u_j(x)$ is the solution to (\ref{eqn:diffuse})
corresponding to the boundary value $g(x) = g_j(x)$. 
Assume that $u_1(x)$ does not vanish inside $\Omega$, then 
 it is easy to verify that $\frac{u_2}{u_1}, \ldots, \frac{u_{n+1}}{u_1}$ satisfy

\begin{equation*}
\left\{\begin{array}{l}
\div\big(\sigma\nabla\tfrac{u_2}{u_1}\big) = 0, \\
\qquad \vdots \\
\div\big(\sigma\nabla\tfrac{u_{n+1}}{u_1}\big) = 0,
\end{array}\right.
\end{equation*}
where $\sigma=Du_1^2$. Otherwise, if 
$\sigma$ is smooth enough, we have

\begin{equation}\label{eqn:sigma}
M \cdot \nabla \ln \sigma = N,\quad 
M = \left(\begin{array}{l}
\big(\nabla\tfrac{u_2}{u_1}\big)^T\\ \quad \vdots \\ \big(\nabla\tfrac{u_{n+1}}{u_1}\big)^T
\end{array}\right),\quad
N = \left(\begin{array}{l}
\Delta\tfrac{u_2}{u_1}\\ \quad \vdots \\ \Delta\tfrac{u_{n+1}}{u_1}
\end{array}\right),
\end{equation}
where $(\cdot)^T$ denotes the transpose of a vector or a matrix.
Since $\frac{u_2}{u_1} = \frac{H_2}{H_1}, \cdots, \frac{u_{n+1}}{u_1} = \frac{H_{n+1}}{H_1}$,
we are able to reconstruct $\nabla\ln\sigma(x)$  locally by solving the linear system (\ref{eqn:sigma})
if the matrix generated by the internal data is nonsingular at $x$. In fact  many internal 
measurements  can be collected in a very short time, and considering 
 $M$ as an  invertible matrix  is indeed a realistic assumption. Notice that  in theory, it  is always possible
to reconstruct $\sigma(x)$ by solving only one linear steady state transport equation  \cite{bonnetier2019stability,triki2020inverse}.   
However, in practice the measurements are noisy, and the transport speed  can posses  critical points with  large multiplicity values  which may generate  severe instabilities in the inversion. 
 
\vskip .3cm
Here we set up a threshold based on the estimation of the noise level and
we formulate the linear system (\ref{eqn:sigma}) in the region where $H_1(x)$ is larger than the threshold.
Such region should be near the set  where $g_1(x)$  is large (the impact zone)  which  mathematically 
is a consequence of the maximum principle and Harnack's inequality \cite{gilbarg2015elliptic}.
Indeed, we shall provide sufficient theoretical conditions that are at the same time consistent with 
experimental observations,  to guarantee the existence of a subregion in which $M$ is nonsingular in Theorem \ref{the:determinant}.
%%%%%%%%%%%%%%%%%%%%%%%%%%%%%%%%
\begin{theorem}\label{the:determinant}
Let $\Omega$ be a $C^{2,1}$ domain, $x^\star\in \partial\Omega$ and $\big(\gamma(x^\star), \tau_1(x^\star), \ldots, \tau_{n-1}(x^\star)\big)$
be the curvilinear coordinates at $x^\star$. Let $g_1, \ldots, g_{n+1}\in C^{2,1}(\partial\Omega)$ be the boundary illuminations satisfying
\begin{itemize}
\item $\eta^{-1} < g_1(x)$ for all $x\in \partial \Omega$, 
\item $\|g_j\|_{C^{2,1}(\partial\Omega)} \leq \eta$,$j = 1, \ldots, n+1$. 
\end{itemize}
Denote $h_j = g_{j+1}/g_1$, $j = 1, \ldots, n$. We further assume
\begin{itemize}
\item $h_1(x)<h_1(x^\star)$ for all $x\in \partial\Omega\setminus \{x^\star\}$, 
\item $\det \big(\nabla_\tau h_2(x^\star), \ldots, \nabla_\tau h_n(x^\star)\big)>\varepsilon$,
\end{itemize}
where $\varepsilon>0$ is a fixed constant and $\nabla_\tau h(x^\star) = 
\big(\nabla h(x^\star)\cdot \tau_1(x^\star), \ldots, \nabla h(x^\star)\cdot \tau_{n-1}(x^\star)\big)^T$.
Consider the set of coefficient 

$$\mathcal{D} = \left\{ (D,\mu) \big| D\in C^2(\overbar{\Omega}), \mu\in C^1(\overbar{\Omega}),
D(x)\geq K^{-1}, \mu(x) \geq K^{-1}, \|D\|_{C^2(\overbar{\Omega})}\leq K,\|\mu\|_{C^1(\overbar{\Omega})} \leq K \right\}$$ with a constant $K>1$. 
Then for $u_j>\eta^{-1}, \, j=1, \cdots n+1,$ in  $ \Omega$,   there exist constants $r_0= r_0(\Omega,n,\mathcal{D},\eta,\varepsilon)>0$ and $C=C(\Omega,n,\mathcal{D},\eta,\varepsilon)\geq 1$ such that $v_j= \frac{u_{j+1}}{u_1},   j=1, \cdots n,$ satisfy

\begin{eqnarray} \label{eqn:determinant}
\det \big( \nabla v_1(x), \ldots, \nabla v_n(x) \big) \ge C^{-1},\\
\left\| \big( \nabla v_1(x), \ldots, \nabla v_n(x) \big)^{-1}\right\|_F \leq C,
\end{eqnarray}
for all $x\in B_{r_0}(x^\star)\cap \Omega$, where  $\|\cdot\|_F$ denotes the Frobenius norm of the matrix.
\end{theorem}
%%%%%%%%%%%%%%%%%%%%%%%%%%%%%%%
\begin{proof} We deduce from classical elliptic regularity that $u_j\in C^{2,1}(\overline{\Omega})$\cite[Theorem 6.14]{gilbarg2015elliptic}.
The maximum principle implies that the minimum of $u_1(x)$ is achieved 
on $\partial \Omega$. That is  $u_1\geq  \min_{\partial \Omega} g_1 >  \eta^{-1}$  in  $ \Omega$. 
Since $v_1$ satisfies

\begin{equation*}
\left\{\begin{array}{ll}
\div\big(Du_1^2\nabla v_1\big) = 0 & \text{in }\Omega,\\
v_1 = h_1& \text{on } \partial\Omega,
\end{array}\right.
\end{equation*}
we again use the maximum principle to obtain  $\max_{\overbar{\Omega}}v_1(x) = h_1(x^\star)$ \cite{gilbarg2015elliptic, le2018nonlinear}.  On the other hand Hopf-Oleinik Lemma implies that
$\nabla_{\nu} v_1(x^\star) > 0$. Actually, one can  show by  contradiction and 
compactness arguments  that  there exists a constant $c_0 = c_0(\Omega,n, \mathcal{D}, \eta)>0$ such that 
$\nabla_{\nu} v_1(x^\star) \geq  c_0$.
Since $h_1(x)$ reaches its maximum at $x^\star$, we have $\nabla_\tau h_1(x^\star)=0,$
and therefore

\begin{eqnarray*}
f(x^\star) & = & \det \big( \nabla v_1(x^\star), \ldots, \nabla v_n(x^\star) \big) \\
       & = & \det \left(\begin{array}{ccc} \nabla_\nu  v_1(x^\star) & \cdots & \nabla_\nu  v_n(x^\star)\\
										   \nabla_\tau h_1(x^\star) & \cdots & \nabla_\tau h_n(x^\star)\end{array}\right)\\
	   & = & \nabla_\nu v_1(x^\star) \det \big(\nabla_\tau h_2(x^\star), \ldots, \nabla_\tau h_n(x^\star)\big)\\
	   & > & c_0\varepsilon.
\end{eqnarray*}
 Since  $v\in C^{2,1}(\overbar{\Omega})$, we have  $f(x)\in C^{1,1}(\overbar{\Omega})$. The strict positivity and differentiability of $f(x^\star)$  imply the existence of the region $B_{r_0}(x^\star)\cap\Omega$  where  the inequality \ref{eqn:determinant} is fulfilled.
\end{proof}

%%%%%%%%%%%%%%%%%%%%%%%%%%%%%%%%
\begin{note}
In Theorem \eqref{the:determinant}, to guarantee the non-singularity of $M$ locally, we only impose  conditions on $g_j$ and the coefficients. The subregion might be very small since we  use the continuity and the non-singularity of $M$ on one point that belongs to the boundary.
In practice, we setup a threshold and compare the condition number of $M$ with it to detect the non-singular region inside which we are able to do the reconstruction.
%Since for most biological tissue the absorption coefficient is much larger than the diffusion coefficient, the non-singularity of $M$ is easy to be satisfied by just shifting the position of the impact zone of the boundary illumination $g(x)$; refer to the numerical experiments in Section \ref{sec:numeric}.
\end{note}
%%%%%%%%%%%%%%%%%%%%%%%%%%%%%%%%

Substituting $u_1 = \frac{\sqrt{\sigma}}{\sqrt{D}}$ into (\ref{eqn:diffuse}) results an equation for $\sqrt{D}$

\begin{equation}\label{eqn:D}
-\Delta \sqrt{D}+\frac{\Delta\sqrt{\sigma}}{\sqrt{\sigma}}\sqrt{D} = \frac{H_1}{\sqrt{\sigma}}.
\end{equation}
To be able to handle this problem theoretically, we assume that we know $\sqrt{D}$ and $\partial_\nu\sqrt D$ on a part of the boundary.
That is, we formulate a Cauchy problem for $\sqrt{D}$. Numerically, we propose a much easier method.
We complete the missing region with background value or averaging of existing value and solve (\ref{eqn:D}) once to reconstruct $\sqrt{D}$.
This simple idea works very well for our numerical experiments and the results are very accurate with small relative errors; see Section \ref{sec:numeric} for more details.

%%%%%%%%%%%%%%%%%%%%%%%%%%%%%%%%%%%%%%%%%%%%%%%%%%%%%%%
\subsection{H\"older stability of the Cauchy problem}\label{sec:stability_cauchy}
Let $\Omega$ be a bounded domain of $\mathbb{R}^n$ with Lipschitz boundary $\partial\Omega$.
Consider the operator $L = -\div (a\nabla\cdot)+b\cdot$ and the Cauchy problem

\begin{equation}\label{eqn:cauchy}
\left\{\begin{array}{ll}
Lv = f   & \text{ in }\Omega,\\
v  = g   & \text{ on }\Gamma,\\
a\partial_{\nu}v = h & \text{ on }\Gamma,
\end{array}\right.
\end{equation}
where $\Gamma$ is a subset of $\partial\Omega$.
We assume that there exists $K>0$ so that 

$$ a(x),b(x)>K^{-1} \text{ for all } x\in\Omega , \text{ and } \|a\|_{C^1(\Omega)},\|b\|_{C(\Omega)}<K.$$ 
Pick $\psi\in C^2(\overbar{\Omega})$ without critical points in $\Omega$ and let $\varphi =e^{\lambda\psi}$.
Let us recall the Carleman estimate for elliptic operators \cite{choulli2016applications, choulli2015triki}.

\begin{proposition} [Carleman inequality]
There exist three strictly positive constants $C$, $\lambda_0$ and $\tau_0$, which depend only on
$\psi, \Omega$ and $K$, so that

\begin{equation}
C\int_{\Omega}\left(\lambda^4\tau^3\varphi^3v^2+\lambda^2\tau\varphi|\nabla v|^2\right)e^{2\tau\varphi} dx 
\le \int_{\Omega}(Lv)^2e^{2\tau\varphi}dx
   +\int_{\partial\Omega}\left(\lambda^3\tau^3\varphi^3v^2+\lambda\tau\varphi|\nabla v|^2\right)e^{2\tau\varphi}d\sigma
\end{equation}
for all $v\in H^2(\Omega), \lambda\ge\lambda_0$ and $\tau\ge\tau_0$.
\end{proposition}

A H\"{o}lder stability of the Cauchy problem (\ref{eqn:cauchy}) has been proved in \cite{alessandrini2009stability,choulli2016applications}.
It has been shown that the $L_2$ norm of the solution $v$ in a closed subregion of $\Omega$ can be bounded by a constant times
terms corresponding to the Cauchy data, the source and the a prior estimation of the solution.
The constant goes to infinity while the closed subregion approaches $\Gamma$ which is a contradiction to intuition.
In the following, we study the stability problem near the boundary $\Gamma$.
\vskip .3cm

We assume that the boundary $\partial\Omega$ satisfies the uniform exterior sphere property (\textbf{UESP}), i.e.,
there exists $\rho>0$ so that, for any $\tilde{x}\in\partial\Omega$, $\exists x_0\in\mathbb{R}\setminus\Omega$ satisfies

$$B(x_0,\rho)\cap\Omega = \emptyset \text{ and } \overbar{B}(x_0,\rho) \cap \overbar{\Omega} = \{\tilde{x}\}.$$
From now on, we fix $\tilde{x}$ to be in the interior of $\Gamma$. 
Let us denote $\Omega(d)=B(x_0,\rho+d)\cap\Omega$. 
The setup of the problem is demonstrated in Figure \ref{fig:cauchy}.
Here $\delta\in (0,1)$ and $\theta\in(0,1)$ are constants independent of $r$.
Since $\tilde{x}$ is in the interior of $\Gamma$, there exists a constant $r_0$ such that, for all $r+\delta<r_0$,
we have $\partial\Omega\cap\partial\Omega(r+\delta) \in \Gamma$.
We will give an upper bound of the solution inside $\Omega(\theta r)$ and study the asymptotic property when $r\rightarrow 0$
in Theorem \ref{the:cauchy_curve}. For the rest of the paper, we use $C$ to denote a constant which may vary from formula to formula
and we will clarify its dependence if necessary. We fix $\lambda\equiv\lambda_0>1$.

\begin{figure}[t]
\centering
\includegraphics[scale=.25]{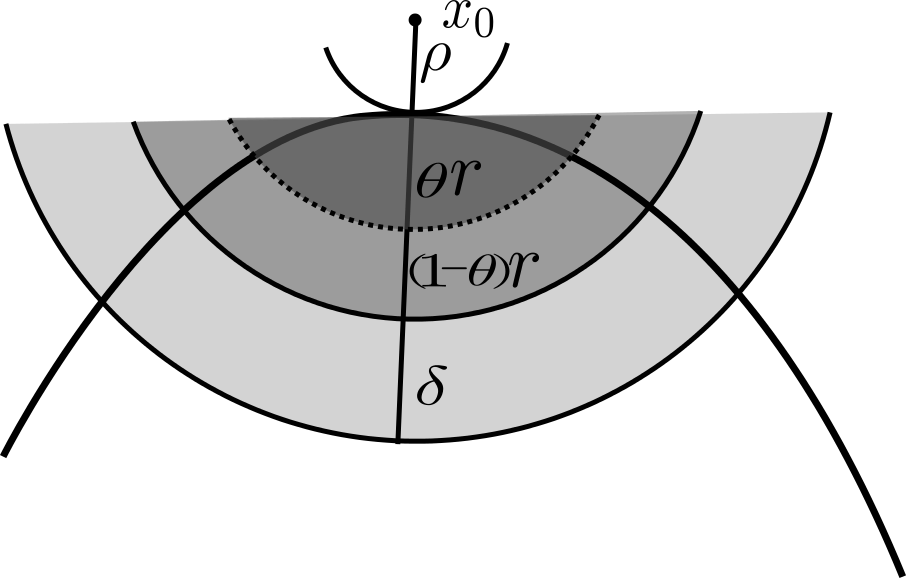}
\caption{Cauchy problem near the boundary of measurement.}
\label{fig:cauchy}
\end{figure}

\begin{theorem}\label{the:cauchy_curve}
There exist two constants $C>0$ and $0<\gamma(r)<1$
so that, for any $v\in H^2(\Omega)$ satisfying (\ref{eqn:cauchy})
with the prior estimation 

$$\|v\|_{H^1(\Omega(r_0))}^2\leq K' \text{ and } \|f\|^2_{L^2(\Omega(r_0))}+\|v\|^2_{L^2(\Gamma)}+\|\nabla v\|^2_{L^2(\Gamma)^n}\leq K',$$
we have

$$C\delta^4 \|v\|^2_{L^2(\Omega(\theta r))} \leq \Big(\|f\|^2_{L^2(\Omega(r_0))}
+\|v\|^2_{L^2(\Gamma)}+\|\nabla v\|^2_{L^2(\Gamma)^n}\Big)^{\gamma(r)},$$
where C is independent of $\delta$ and $r$.
Moreover, a possible choice of $\gamma$ is

\begin{equation}\label{eqn:gamma}
\gamma = \frac{\big((\rho+r)^{2\lambda_0}-(\rho+\theta r)^{2\lambda_0}\big)\rho^{2\lambda_0}}
{\big((\rho+r)^{2\lambda_0}-\rho^{2\lambda_0}\big)(\rho+\theta r)^{2\lambda_0}},
\end{equation}
which is a decreasing function of $r$ for

\begin{equation}\label{eqn:bound_r}
r< \min\left\{
\frac{1}{\theta}, \frac{3(1-\theta)}{(2\lambda_0-1)\big(4^{\lambda_0-1}-1\big)}
\right\}\rho,
\end{equation}
and converges to $1-\theta$ as
$r\rightarrow 0$.
\end{theorem}

\begin{proof}
Define

$$\psi(x) = \ln\big((\rho+r_0)^2/|x-x_0|^2\big).$$ 
Then 

$$|\nabla\psi(x)| = \frac{2}{|x-x_0|} \ge \frac{2}{\rho+r_0}>0,\ \text{ for all } x\in\Omega(r+\delta).$$
That is to say, $\Psi$ satisfies the non-critical-point condition.

Let $\chi\in C^\infty(\overbar\Omega)$, $\chi=1$ in $\overbar\Omega(r)$ and $\chi=0$ in $\overbar{\Omega}\setminus\overbar{\Omega}(r+\delta)$. 
Therefore $\partial^{\alpha}\chi \leq K"\delta^{-|\alpha|}$, $|\alpha|\le 2$,
where $K"$ is a constant independent of $\delta$.
Applying the Carleman inequality to $u=\chi v$ in $\Omega(r+\delta)$, we obtain

\begin{equation}\label{eqn:carleman_bound}
C\int_{\Omega(\theta r)} |v|^2 e^{2\tau\varphi}dx 
\leq \int_{\Omega(r+\delta)} (L(\chi v))^2 e^{2\tau\varphi}dx 
	+\int_{\Gamma} \big((\chi v)^2+|\nabla(\chi v)|^2\big)e^{2\tau\varphi} d\sigma,
\end{equation} 
where 

$$\varphi = \frac{(\rho+r_0)^{2\lambda_0}}{|x-x_0|^{2\lambda_0}}$$
and $C$ depends only on $\Omega, K, \rho, r_0, \lambda_0$ and $\tau_0$.

Using $L(\chi v) = \div(a\nabla\chi)v+2a\nabla\chi\cdot\nabla v+\chi f$ and the estimates on $\chi$ and its derivatives,
we obtain

\begin{equation}\label{eqn:carleman_rhs1}
\int_{\Omega(r+\delta)} (L(\chi v))^2 e^{2\tau\varphi}dx \le C\int_{\Omega(r_0)} f^2 e^{2\tau\varphi}dx
	+\frac{C}{\delta^4}\int_{\Omega(r+\delta)\setminus\Omega(r)}\big(v^2+|\nabla v|^2\big)e^{2\tau\varphi} dx,
\end{equation}
where $C$ depends only on $K$ and $K"$.
For the second term of the right hand side of (\ref{eqn:carleman_bound}), we have

\begin{equation}\label{eqn:carleman_rhs2}
\int_{\Gamma} \big((\chi v)^2+|\nabla(\chi v)|^2\big)e^{2\tau\varphi} d\sigma \leq 
\frac{C}{\delta^2}\int_{\Gamma} \big(v^2+|\nabla v|^2\big)e^{2\tau\varphi} d\sigma,
\end{equation}
where $C$ depends only on $K"$.
Combining (\ref{eqn:carleman_bound})-(\ref{eqn:carleman_rhs2}) results

\begin{equation}\label{eqn:bound_L2v}
C\delta^4\int_{\Omega(\theta r)} |v|^2 e^{2\tau\varphi}dx \leq 
\int_{\Omega(r_0)} f^2 e^{2\tau\varphi}dx+\int_{\Gamma} \big(v^2+|\nabla v|^2\big)e^{2\tau\varphi}d\sigma
+\int_{\Omega(r+\delta)\setminus\Omega(r)}\big(v^2+|\nabla v|^2\big)e^{2\tau\varphi} dx,
\end{equation}
where $C$ depends only on $\Omega, K, K", \rho, r_0, \lambda_0$ and $\tau_0$.

Define

$$\varphi_0 = \frac{(\rho+r_0)^{2\lambda_0}}{(\rho+\theta r)^{2\lambda_0}},\ \ 
  \varphi_1 = \frac{(\rho+r_0)^{2\lambda_0}}{(\rho+r)^{2\lambda_0}}\ \text{ and }\  
  \varphi_2 = \frac{(\rho+r_0)^{2\lambda_0}}{\rho^{2\lambda_0}}.$$
Then $\varphi \ge \varphi_0$ in $\Omega(\theta r)$, $\varphi \le \varphi_1$ in $\Omega(r+\delta)\setminus\Omega(r)$
and  $\varphi \le \varphi_2$ in $\overbar\Omega$.
Substituting these estimations into (\ref{eqn:bound_L2v}) results

\begin{equation}\label{eqn:upbound1}
C\delta^4 \|v\|^2_{L^2(\Omega(\theta r))} \leq e^{-\tau\alpha(r;\theta)}\|v\|^2_{H^1(\Omega(r_0))}
	+e^{\tau\beta(r;\theta)}\Big(\|f\|^2_{L^2(\Omega(r_0))}+\|v\|^2_{L^2(\Gamma)}+\|\nabla v\|^2_{L^2(\Gamma)^n}\Big),
\end{equation}
where

$$\alpha(r;\theta) = 2(\varphi_0-\varphi_1),\ \beta(r;\theta) = 2(\varphi_2-\varphi_0). $$
For simplicity, let us denote

$$\mathcal{A} = \|v\|^2_{H^1(\Omega(r_0))},\ 
\mathcal{B} = \|f\|^2_{L^2(\Omega(r_0))}+\|v\|^2_{L^2(\Gamma)}+\|\nabla v\|^2_{L^2(\Gamma)^n}
\text{ and } \mathcal{F}(\tau) = e^{-\tau\alpha}\mathcal{A}+e^{\tau\beta}\mathcal{B}.$$

By calculating the derivative, it is easy to show that $\mathcal{F}(\tau)$ first decreases and then increases as
$\tau$ goes from $0$ to infinity, and $\mathcal{\mathcal{F}}$ obtains its minimum at 

$$\tilde\tau=\frac{\ln\frac{\alpha \mathcal{A}}{\beta \mathcal{B}}}{\alpha+\beta}.$$
If $\tau_0\leq\tilde\tau$, we can take $\tau=\tilde\tau$, and in this case

\begin{equation}\label{eqn:Fcase1}
\mathcal{F}(\tilde{\tau}) = \Big(\big(\tfrac{\alpha}{\beta}\big)^{-\frac{\alpha}{\alpha+\beta}}+
\big(\tfrac{\alpha}{\beta}\big)^{\frac{\beta}{\alpha+\beta}}\Big)
\mathcal{A}^{\frac{\beta}{\alpha+\beta}}\mathcal{B}^{\frac{\alpha}{\alpha+\beta}}.
\end{equation}
If $\tau_0>\tilde\tau$, that is, $e^{-\tau_0\alpha}\mathcal{A}<\frac{\beta}{\alpha}e^{\tau_0\beta}\mathcal{B}$, we have

\begin{equation}\label{eqn:Fcase2}
\mathcal{F}(\tau_0) \leq \big(1+\tfrac{\beta}{\alpha}\big)e^{\tau_0\beta}
\mathcal{B}^{\frac{\beta}{\alpha+\beta}}\mathcal{B}^{\frac{\alpha}{\alpha+\beta}}.
\end{equation}
To obtain a suitable upper bound of (\ref{eqn:Fcase1}) and (\ref{eqn:Fcase2}) independent of $r$,
we will study the monotonicity and bound of $\beta$, $\frac{\alpha}{\beta}$ and $\frac{\alpha}{\alpha+\beta}$ as $r\rightarrow 0$.

By the mean value theorem, there exists $\eta\in (0,1)$ such that

\begin{eqnarray*}
\beta & = & 2(\rho+r_0)^{2\lambda_0}\left(\frac{1}{\rho^{2\lambda_0}}-\frac{1}{(\rho+\theta r)^{2\lambda_0}}\right)\\
      & = & \theta r\frac{4\lambda_0(\rho+r_0)^{2\lambda_0}}{\big(\rho+\theta(1-\eta)r\big)^{2\lambda_0+1}},
\end{eqnarray*}
and therefore

$$\beta \leq \theta r_0\frac{4\lambda_0(\rho+r_0)^{2\lambda_0}}{\rho^{2\lambda_0+1}}.$$

Since

$$\frac{\alpha}{\beta} = \frac{1-\tfrac{\varphi_1}{\varphi_0}}{\tfrac{\varphi_2}{\varphi_0}-1} 
= \frac{1-\big(\tfrac{\rho+\theta r}{\rho+r}\big)^{2\lambda_0}}{\big(\tfrac{\rho+\theta r}{\rho}\big)^{2\lambda_0}-1},$$
it is easy to show

$$\lim_{r\rightarrow 0}\frac{\alpha}{\beta} = \frac{1}{\theta}-1.$$
The derivative of $\frac{\alpha}{\beta}$ satisfies

\begin{eqnarray*}
\Big(\frac{\alpha}{\beta}\Big)' & \propto & 2\lambda_0\Big(\tfrac{\rho+\theta r}{\rho+r}\Big)^{2\lambda_0-1}\tfrac{(1-\theta)\rho}{(\rho+r)^2}\Big(\Big(\tfrac{\rho+\theta r}{\rho}\Big)^{2\lambda_0}-1\Big)
-2\lambda_0\Big(\tfrac{\rho+\theta r}{\rho}\Big)^{2\lambda_0-1}\tfrac{\theta}{\rho}
\Big(1-\Big(\tfrac{\rho+\theta r}{\rho+r}\Big)^{2\lambda_0}\Big),\\
& \propto & \big((\rho+\theta r)^{2\lambda_0}\rho-\rho^{2\lambda_0+1}\big)(1-\theta)-
\big((\rho+r)^{2\lambda_0+1}-(\rho+\theta r)^{2\lambda_0}(\rho+r)\big)\theta.
\end{eqnarray*}
By Taylor expansions, there exist $\xi_1, \xi_2, \xi_3 \in (0,1)$ such that

\begin{eqnarray*}
(\rho+\theta r)^{2\lambda_0} & = & \rho^{2\lambda_0}+2\lambda_0\rho^{2\lambda_0-1}\theta r+\lambda_0(2\lambda_0-1)\rho^{2\lambda_0-2}\theta^2 r^2\\
& & +\frac{1}{3}\lambda_0(2\lambda_0-1)(2\lambda_0-2)(\rho+\xi_1\theta r)^{2\lambda_0-3}\theta^3 r^3,\\
(\rho+r)^{2\lambda_0+1} & = & \rho^{2\lambda_0+1}+(2\lambda_0+1)\rho^{2\lambda_0}r+(2\lambda_0+1)\lambda_0\rho^{2\lambda_0-1}r^2\\
& & +\frac{1}{3}(2\lambda_0+1)\lambda_0(2\lambda_0-1)(\rho+\xi_2 r)^{2\lambda_0-2} r^3,\\
(\rho+\theta r)^{2\lambda_0} & = & \rho^{2\lambda_0}+2\lambda_0\rho^{2\lambda_0-1}\theta r
+\lambda_0(2\lambda_0-1)(\rho+\xi_3\theta r)^{2\lambda_0-2}\theta^2 r^2.
\end{eqnarray*}
By substituting $\xi_2 =0$ and $\xi_1=\xi_3=1$, we obtain the following upper bound,

\begin{eqnarray*}
\Big(\frac{\alpha}{\beta}\Big)' & \propto & \big((\rho+\theta r)^{2\lambda_0}\rho-\rho^{2\lambda_0+1}\big)+
\big(\rho^{2\lambda_0+1}-(\rho+r)^{2\lambda_0+1}+(\rho+\theta r)^{2\lambda_0}r\big)\theta,\\
& < & -(2\lambda_0+1)\lambda_0\rho^{2\lambda_0-1}(1-\theta)\theta r^2
	+\frac{1}{3}(2\lambda_0+1)\lambda_0(2\lambda_0-1)\big((\rho+\theta r)^{2\lambda_0-2}-\rho^{2\lambda_0-2}\big)\theta r^3.
\end{eqnarray*}
To obtain $\big(\frac{\alpha}{\beta}\big)'<0$, we only need

\begin{equation*}
\frac{1}{3}(2\lambda_0+1)\lambda_0(2\lambda_0-1)\big((\rho+\theta r)^{2\lambda_0-2}-\rho^{2\lambda_0-2}\big)\theta r^3
\ <\ (2\lambda_0+1)\lambda_0\rho^{2\lambda_0-1}(1-\theta)\theta r^2,
\end{equation*}
which is

\begin{equation}\label{eqn:bound_r1}
r < \frac{3(1-\theta)\rho}{(2\lambda_0-1)\Big(\big(1+\theta\frac{r}{\rho}\big)^{2\lambda_0-2}-1\Big)}.
\end{equation}
Assume that $r<\frac{\rho}{\theta}$. Then (\ref{eqn:bound_r1}) is satisfied if 

$$r < \frac{3(1-\theta)}{(2\lambda_0-1)\big(4^{\lambda_0-1}-1\big)}\rho.$$
To summarize, for

\begin{equation*}
r< \min\left\{
\frac{1}{\theta}, \frac{3(1-\theta)}{(2\lambda_0-1)\big(4^{\lambda_0-1}-1\big)}
\right\}\rho
\end{equation*}
$\tfrac{\alpha}{\beta}$ is a decreasing function of $r$ and 

$$\lim_{r\rightarrow 0}\frac{\alpha}{\beta} = \frac{1}{\theta}-1.$$
Therefore, $\tfrac{\alpha}{\alpha+\beta}$ is also a decreasing function of $r$ and 

$$\lim_{r\rightarrow 0}\frac{\alpha}{\alpha+\beta} = 1-\theta.$$

We have proved that $\beta, \tfrac{\alpha}{\beta}$ and $\tfrac{\alpha}{\alpha+\beta}$ can be lower and upper bounded by
positive constants depending only on $\theta, \rho, K', \lambda_0$ and $\tau_0$ if $r$ satisfies (\ref{eqn:bound_r}).
Combining (\ref{eqn:upbound1})-(\ref{eqn:Fcase2}), we have the conclusion 

\begin{equation}
C\delta^4 \|v\|^2_{L^2(\Omega(\theta r))} \leq \Big(\|f\|^2_{L^2(\Omega(r_0))}+\|v\|^2_{L^2(\Gamma)}+\|\nabla v\|^2_{L^2(\Gamma)^n}\Big)^{\gamma(r)},
\end{equation}
where $C$ is independent of $r$ and $\delta$.
For $r$ satisfying (\ref{eqn:bound_r}), 
 $\gamma(r)=\frac{\alpha}{\alpha+\beta}$ is a decreasing function of $r$,
which converges to $1-\theta$ as $r\rightarrow 0$.
\end{proof}

%%%%%%%%%%%%%%%%%%%%%%%%%%%%%%%%%%%%%%%%%%%%
\subsection{H\"older stability to reconstruct $D$ and $\mu$} \label{sec:stability_Dmu}
Assume that $\Omega$ is a bounded domain with $C^{3,1}$ boundary.
Let us consider a set of coefficients

$$\mathcal{C} = \left\{(D,\mu)\ \big\vert\ D(x),\ \mu(x)> \kappa^{-1} \text{ for all } x\in\overbar{\Omega}, \quad
\|D\|_{C^{2,1}(\overbar{\Omega})},\ \|\mu\|_{C^3(\overbar{\Omega})}<\kappa \right\}$$
for a constant $\kappa>1$, a set of boundary conditions

$$\mathcal{G}=\left\{g|g\in C^{3,1}(\partial\Omega),\ g(x)\ge 0 \text{ for all } x\in\partial\Omega\right\},$$
and a subregion $\Omega(r_0)$ defined in the previous section with $\Gamma:=\partial\Omega\cap\partial\Omega(r_0)$. For $g\in \mathcal{G},$  we deduce from Shauder elliptic 
regularity  \cite[Theorem 6.14 and 6.19]{gilbarg2015elliptic}, that  (\ref{eqn:diffuse}) has a unique solution $u\in C^3(\overbar\Omega)$. 
\vskip .3cm
To study the stability of the inverse problem, we choose $(D,\mu)$ and $(D',\mu')$ from $\mathcal{C}$ and 
solve (\ref{eqn:diffuse}) with the same boundary conditions $\big\{g_j\big\}_{j=1}^{n+1}\subset \mathcal{G}$ to obtain
$\big\{H_j\big\}_{j=1}^{n+1}$ and $\big\{H'_j\big\}_{j=1}^{n+1}$ respectively.
Moreover, we assume that $D\equiv D'$ and $\partial_\nu D \equiv \partial_\nu D'$ on $\Gamma$.
The following lemma provides a piecewise stability estimation to reconstruct functions related to $\sigma(x)$.

\begin{lemma}\label{lem:sigma} Let $c>1$ be fixed,  and assume that $H_1(x), H'_1(x) > c $, $M(x)$ and $M'(x)$  defined in \eqref{eqn:sigma} are invertible, and 
$\|M^{-1}(x)\|_F, \|M(x)\|_F \leq c, $  for all $x\in \overbar\Omega(r_0)$,
where $\|\cdot\|_F$ denotes the Frobenius norm of the matrix.
Assume also that $H_1(x), H'_1(x) > c > 0$ for a constant $c$ and for all $x\in \Omega(r_0)$.
Then there exists a strictly positive constant $C$ such that

\begin{eqnarray}
C\|\nabla\ln\sigma(x)-\nabla\ln\sigma'(x)\|_2^2 & \leq & 
	\sum_{j=1}^{n+1}\|H_j-H'_j\|^2_{C^2(\overbar{\Omega}(r_0))},\label{eqn:bound_sigma1}\\
C\big|\div\big(\nabla\ln\sigma(x)-\nabla\ln\sigma'(x)\big)\big|^2 & \leq &
	\sum_{j=1}^{n+1}\|H_j-H'_j\|^2_{C^3(\overbar{\Omega}(r_0))}.\label{eqn:bound_sigma2}
\end{eqnarray}
Meanwhile, if 

$$\sum_{j=1}^{n+1}\|H_j-H'_j\|^2_{C^2(\overbar{\Omega}(r_0))}\ll 1,$$
we also have

\begin{equation}\label{eqn:bound_sigma3}
C|\sigma(x)-\sigma'(x)|^2 \ \leq \ \sum_{j=1}^{n+1}\|H_j-H'_j\|^2_{C^2(\overbar{\Omega}(r_0))}.
\end{equation}
\end{lemma}
\begin{proof}
Since

$$\left\{\begin{array}{lcl}
M(x)\cdot \nabla\ln\sigma(x) & = & N(x), \\ M'(x)\cdot \nabla\ln\sigma'(x) & = & N'(x),
\end{array}\right.$$
we have

\begin{equation}\label{eqn:nabla_ln_sig_equal}
\nabla\ln\sigma(x)-\nabla\ln\sigma'(x) = -M^{-1}(x)\big(M(x)-M'(x)\big)\nabla\ln\sigma'(x)+M^{-1}(x)\big(N(x)-N'(x)\big)
\end{equation}
and therefore

\begin{equation}\label{eqn:nabla_ln_sig_bound}
C\|\nabla\ln\sigma(x)-\nabla\ln\sigma'(x)\|^2_2 \ \le \ \|M(x)-M'(x)\|_F^2 + \|N(x)-N'(x)\|_2^2.
\end{equation}
Let us recall

$$M(x)-M'(x) = \left(\begin{array}{c}
\nabla\Big(\tfrac{H_2(x)}{H_1(x)}-\tfrac{H'_2(x)}{H'_1(x)}\Big)^T\\ \quad \vdots \\ 
\nabla\Big(\tfrac{H_{n+1}(x)}{H_1(x)}-\tfrac{H'_{n+1}(x)}{H'_1(x)}\Big)^T
\end{array}\right)\ \text{ and } \ 
N(x)-N'(x) = \left(\begin{array}{c}
\Delta\Big(\tfrac{H_2(x)}{H_1(x)}-\tfrac{H'_2(x)}{H'_1(x)}\Big)\\ \quad \vdots \\ 
\Delta\Big(\tfrac{H_{n+1}(x)}{H_1(x)}-\tfrac{H'_{n+1}(x)}{H'_1(x)}\Big)
\end{array}\right).$$
Therefore we have the following estimations

\begin{eqnarray*}
C\|M(x)-M'(x)\|_F^2 & \leq & \sum_{j=1}^{n+1}\|H_j-H'_j\|^2_{C^1(\overbar{\Omega}(r_0))},\\
C\|N(x)-N'(x)\|_2^2 & \leq & \sum_{j=1}^{n+1}\|H_j-H'_j\|^2_{C^2(\overbar{\Omega}(r_0))}.
\end{eqnarray*}
Combining with (\ref{eqn:nabla_ln_sig_bound}), we obtain (\ref{eqn:bound_sigma1}).
Taking the divergence of (\ref{eqn:nabla_ln_sig_equal}), and following the same procedure,
we arrive at (\ref{eqn:bound_sigma2}).

Through integration along a curve $l\in \Omega(r_0)$ connecting a boundary point $x_0\in\Gamma$ and $x$, we obtain

$$\left\{\begin{array}{lcl}
\ln\sigma(x) & = & \int_l\nabla\ln\sigma(x)\cdot d\v{l}+\ln\sigma(x_0), \\ 
\ln\sigma'(x) & = & \int_l\nabla\ln\sigma'(x)\cdot d\v{l}+\ln\sigma'(x_0). 
\end{array}\right.$$
Subtracting one equality by another results

$$C|\ln\sigma(x)-\ln\sigma'(x)|^2\ \leq \ \sum_{j=1}^{n+1}\|H_j-H'_j\|^2_{C^2(\overbar{\Omega}(r_0))}.$$
Since $\sigma-\sigma' = \sigma\big(1-e^{\ln\sigma'-\ln\sigma}\big)$, 
we can apply the Taylor expansion of $e^x$ to obtain (\ref{eqn:bound_sigma3}) if 

$$\sum_{j=1}^{n+1}\|H_j-H'_j\|^2_{C^2(\overbar{\Omega}(r_0))}\ll 1.$$
\end{proof}
This lemma leads to the main result of this paper.
\begin{theorem}\label{the:bound_D}
Let us choose $(D,\mu)$ and $(D',\mu')$ from $\mathcal{C}$ such that $D\equiv D'$ and $\partial_\nu D \equiv \partial_\nu D'$ on $\Gamma$.
Assume that the data set $\big\{H_j\big\}_{j=1}^{n+1}$ and $\big\{H'_j\big\}_{j=1}^{n+1}$ satisfy all assumptions in Lemma \ref{lem:sigma}.
Then we have

\begin{equation}\label{eqn:bound_D}
C\|D-D'\|_{L^2(\Omega(\theta r))} \ \leq \ \left(\sum_{j=1}^{n+1}\|H_j-H'_j\|^2_{C^3(\overbar{\Omega}(r_0))}\right)^{\gamma(r)},
\end{equation}
where $\gamma(r)$ is given by (\ref{eqn:gamma}).
\end{theorem}
\begin{proof}
Since $\sqrt{D}$ and $\sqrt{D'}$ satisfy (\ref{eqn:D}) and they have the same boundary value and normal derivative on $\Gamma$,
we can formulate the following Cauchy problem for $\sqrt{D}-\sqrt{D'}$

\begin{equation*}
\left\{\begin{array}{ll}
-\Delta(\sqrt{D}-\sqrt{D'})+\frac{\Delta\sqrt{\sigma}}{\sqrt{\sigma}}\big(\sqrt{D}-\sqrt{D'}\big) = 
	-\Big(\frac{\Delta\sqrt{\sigma}}{\sqrt{\sigma}}-\frac{\Delta\sqrt{\sigma'}}{\sqrt{\sigma'}}\Big)\sqrt{D'}
	+\frac{1}{\sqrt{\sigma}}\big(H_1-H_1'\big)+\Big(\frac{1}{\sqrt{\sigma}}-\frac{1}{\sqrt{\sigma'}}\Big)H'_1  & \text{ in }\Omega,\\
\sqrt{D}-\sqrt{D'} = 0   & \text{ on }\Gamma,\\
\partial_{\nu} \big(\sqrt{D}-\sqrt{D'} = 0\big) = 0 & \text{ on }\Gamma.
\end{array}\right.
\end{equation*}
Since

$$\frac{\Delta\sqrt{\sigma}}{\sqrt{\sigma}} = \frac{1}{4}|\nabla\ln\sigma|^2+\frac{1}{2}\div(\nabla\ln\sigma),$$
applying Theorem \ref{the:cauchy_curve} results

\begin{eqnarray*}
C\big\|\sqrt{D}-\sqrt{D'}\big\|^2_{L^2(\Omega(\theta r))} & \leq &
\left(\big\|H_1-H_1'\big\|^2_{L^2(\Omega(r_0))} 
+ \big\|\nabla\ln\sigma-\nabla\ln\sigma'\big\|^2_{L^2(\Omega(r_0))^n}\right. \\
& & \left. 
+ \big\|\div\big(\nabla\ln\sigma-\nabla\ln\sigma'\big)\big\|^2_{L^2(\Omega(r_0))} 
+ \big\|\sigma-\sigma'\big\|^2_{L^2(\Omega(r_0))}\right)^{\gamma(r)}.
\end{eqnarray*}
Integrating all inequalities in Lemma \ref{lem:sigma} over $\Omega(r_0)$,
we obtain the final estimation (\ref{eqn:bound_D}).
\end{proof}
The stability estimation to reconstruct $\mu$ is given by the next corollary.
\begin{corollary} \label{lastcorollary}
Assume that all assumptions in Theorem \ref{the:bound_D} are satisfied.
Then we have the following H\"older stability estimation to reconstruct $\mu$

\begin{equation}\label{eqn:bound_mu}
C\|\mu-\mu'\|_{L^2(\Omega(\theta r))} \ \leq \ \left(\sum_{j=1}^{n+1}\|H_j-H'_j\|^2_{C^3(\overbar{\Omega}(r_0))}\right)^{\gamma(r)}.
\end{equation}
\end{corollary}
\begin{proof}
Since $u_1 = \tfrac{\sqrt{\sigma}}{\sqrt D}$, combining estimations for $D$ and $\sigma$ results

$$C\|u_1(x)-u'_1(x)\|_{L^2(\Omega(\theta r))}\ \leq \left(\sum_{j=1}^{n+1}\|H_j-H'_j\|^2_{C^3(\overbar{\Omega}(r_0))}\right)^{\gamma(r)}.$$  
Applying the same procedure on $\mu = \tfrac{H_1}{u_1}$ gives the estimation (\ref{eqn:bound_mu}).
\end{proof}

\begin{note} 
The obtained stability results  in Theorem \ref{the:bound_D}, and Corollary \ref{lastcorollary},  indicate that the resolution of PAT is better near the impact zone of the optical illumination sources, and deteriorates far away.
%, which is in agreement with known experimental observations  \cite{wang2017photoacoustic}.
According to Theorem \ref{the:determinant}, it is possible to impose conditions only on the  boundary data $\big\{g_j\big\}_{j=1}^{n+1}, \big\{g'_j\big\}_{j=1}^{n+1}\subset \mathcal{G}$ in order to have all the assumptions in Theorem \ref{the:bound_D} being satisfied. Doing so, we can trace out  the stability constants in Theorem \ref{the:bound_D}, and Corollary \ref{lastcorollary} and show that they only depend  on the boundary data,  $n$, and $\Omega$.  
\end{note}

%%%%%%%%%%%%%%%%%%%%%%%%%%%%%%%%%%%%%%%
\section{Numerical experiments}\label{sec:numeric}

In this section, we will present three numerical experiments, for which we choose $\Omega$ to be the unit disc.
To simulate the internal data $H_j(x), j = 1, 2, 3$, we solve (\ref{eqn:diffuse}) with three different boundary conditions,
each of which is a normal distribution with the standard deviation 0.3 and the peak at the angle
$\frac{4}{9}\pi$, $\frac{1}{2}\pi$ and $\frac{5}{9}\pi$ respectively.

\vskip .3cm
We take the gradient and Laplace of $\frac{H_2}{H_1}$ and $\frac{H_3}{H_1}$ to formulate the linear system (\ref{eqn:sigma}) 
which is solved for each discrete position $x$ whenever it is possible.
Actually the non-singularity of the matrix is a quite weak constraint. At least for all our experiments, we never violate it.
Note that we always smooth the data locally before taking the derivative to alleviate oscillation caused by
noise or numerical discretization.
By integration along a suitable curve from a known boundary point to the unknown point $x$, 
we are able to obtain $\ln\sigma(x)$ and hence $\sigma(x)$.
In practice, we choose to start from 10 different boundary points and average to stabilize the computation.

\vskip .3cm
The next step is to compute the coefficient $\frac{\Delta\sqrt{\sigma}}{\sqrt{\sigma}}$ of the equation (\ref{eqn:sigma}).
Since we have constructed $\sigma$, we can calculate the coefficient directly by taking the Laplace of $\sqrt{\sigma}$.
This procedure is extremely unstable because the error from the numerical integration is 
dramatically amplified during calculating the derivative. Alternatively, we calculate the coefficient by the following equality

$$\frac{\Delta\sqrt{\sigma}}{\sqrt{\sigma}} = \frac{1}{4}|\nabla\ln\sigma|^2+\frac{1}{2}\div(\nabla\ln\sigma).$$
Let us emphasize that we only do the computation in a suitable subregion where $H(x)$ is larger than a threshold
which depends on the estimation of the noise level.

\vskip .3cm
Since the coefficient and source terms of (\ref{eqn:D}) are partially reconstructed,
theoretically to reconstruct $D$ we need to solve the Cauchy problem corresponding to $\sqrt{D}$.
If the medium is homogeneous, we can verify that

$$\frac{H}{\sqrt{\sigma}} = \frac{\mu}{\sqrt{D}}\ \text{ and }\ \frac{\Delta\sqrt{\sigma}}{\sqrt{\sigma}}=\frac{\mu}{\sqrt{D}}.$$
Numerically we complete the missing region by a suitable constant, for example the background value or the average of the known value,
and solve the partial differential equation (\ref{eqn:D}) only once to reconstruct $\sqrt{D}$.
We find out that this strategy only affects the reconstruction in a small area close to the boundary of the completing region.
In the following experiments, we only demonstrate reconstructions in a proper subregion.

\subsection{Synthetic medium -- smooth case}
\label{sec:smooth}

The medium is a unit disc with inhomogeneity composed of rectangles and discs of different size.
The background truth of $D$ and $\mu$ are $0.2$ and $20$ respectively.
The variation of $D$ (resp. $\mu$) varies from $0.1$ to $0.35$ (resp. $10$ to $35$).
To obtain the smooth medium, we take the convolution of the piecewise constant medium from the next experiment with a Gaussian function.
The positions of the inhomogeneity are almost the same for $D$ and $\mu$ except that we intentionally 
remove the rectangle on the top in $D$ and add a triangle in $\mu$; 
refer to Figure \ref{fig:smooth_true_D} and \ref{fig:smooth_true_mu}.

\vskip .3cm
We are able to correctly reconstruct positions and values of the inhomogeneity;
refer to Figure \ref{fig:smooth_recon_D} and \ref{fig:smooth_recon_mu}.
The relative error in the region $y>0.2$ is $3.42\%$ for $D$ and $3.19\%$ for $\mu$.
We could notice from the color of Figure \ref{fig:smooth} that the reconstruction is a little lighter than the truth,
which may be the result of the smoothing technique applied each time before calculating derivatives.

\begin{figure}[ptb]
\centering
\begin{subfigure}[p]{.45\textwidth}
    \centering
   	\includegraphics[scale=.4,trim={0cm 1cm 0cm 1cm},clip]{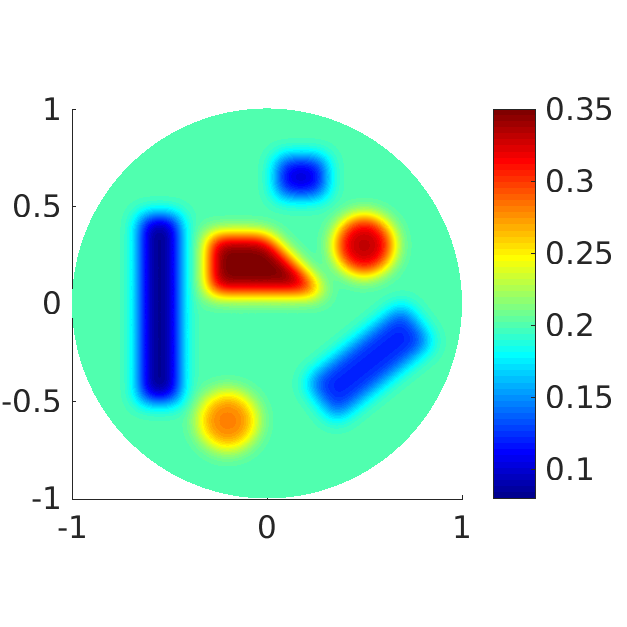}
   	\caption{Background truth of $D$}
   	\label{fig:smooth_true_D}
\end{subfigure}
\begin{subfigure}[p]{.45\textwidth}
    \centering
    \includegraphics[scale=.4,trim={0cm 1cm 0cm 1cm},clip]{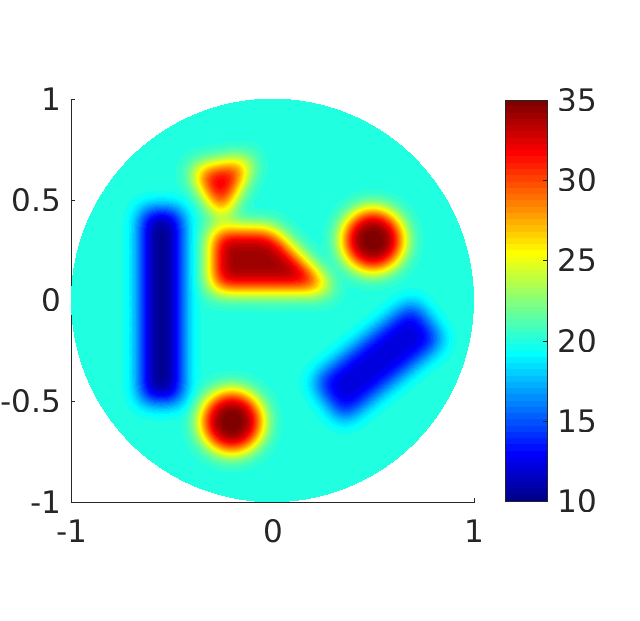}
   	\caption{Background truth of $\mu$}
   	\label{fig:smooth_true_mu}
\end{subfigure}
\begin{subfigure}[p]{.45\textwidth}
    \centering
    \includegraphics[scale=.4,trim={0cm 3cm 0cm 3cm},clip]{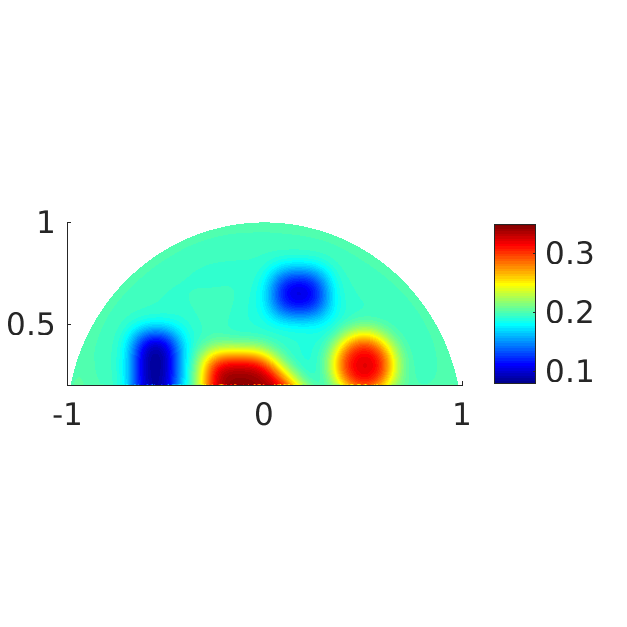}
   	\caption{Reconstruction of $D$}
   	\label{fig:smooth_recon_D}
\end{subfigure}
\begin{subfigure}[p]{.45\textwidth}
    \centering
    \includegraphics[scale=.4,trim={0cm 3cm 0cm 3cm},clip]{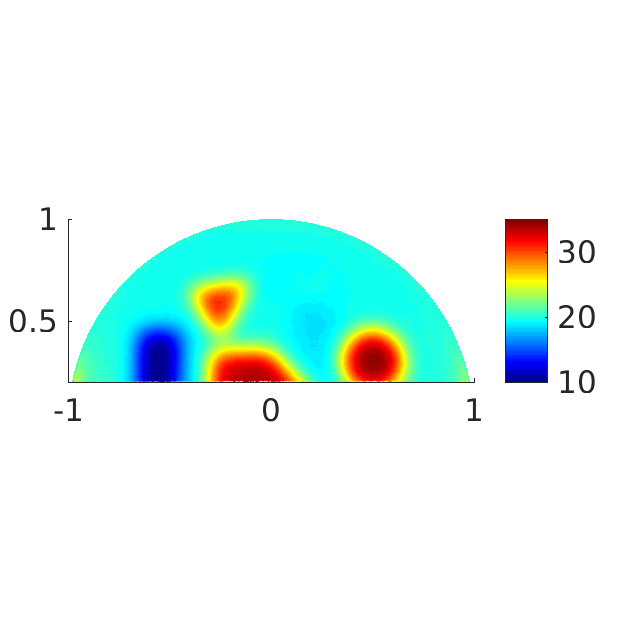}
   	\caption{Reconstruction of $\mu$}
   	\label{fig:smooth_recon_mu}
\end{subfigure}
\caption{Local reconstructions of smooth $D$ and $\mu$.}
\label{fig:smooth}
\end{figure}

\subsection{Synthetic medium -- discontinuous case}
\label{sec:discrete}

To obtain the theoretical stability, it requires certain smoothness of the coefficients $D$ and $\mu$.
In this experiment we try our reconstruction algorithm on a problem with piecewise constant value.
We do not pay any additional attention to the discontinuity or use any special trick inside the code.
The background truth and reconstructed results are demonstrated in Figure \ref{fig:discrete}.
All the embedded inclusions with different shapes are well reconstructed.
The relative error is larger than the smooth case, $8.56\%$ for $D$ and $4.70\%$ for $\mu$ in the region $y>0.2$.

\vskip .3cm
Interestingly, due to the discontinuity of $D$ and $\mu$, 
we observe huge jumps in the reconstructed coefficient $\frac{\Delta\sqrt{\sigma}}{\sqrt{\sigma}}$
on the boundary of the inclusions. The value of the jump depends on the smoothing technique.
That's to say, in the process of filtering out potential noise while computing derivatives,
we decrease the true extreme value of the coefficient as well.
But the huge error in formulating the coefficient and source of (\ref{eqn:D}) does not affect
the reconstruction of $D$ too much, especially qualitatively.

\begin{figure}[ptb]
\centering
\begin{subfigure}[p]{.45\textwidth}
    \centering
   	\includegraphics[scale=.4,trim={0cm 1cm 0cm 1cm},clip]{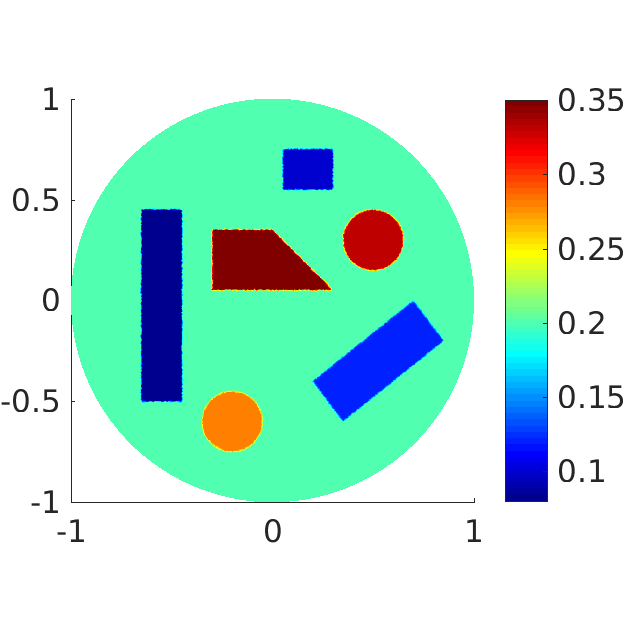}
   	\caption{Background truth of $D$}
   	\label{fig:discrete_true_D}
\end{subfigure}
\begin{subfigure}[p]{.45\textwidth}
    \centering
    \includegraphics[scale=.4,trim={0cm 1cm 0cm 1cm},clip]{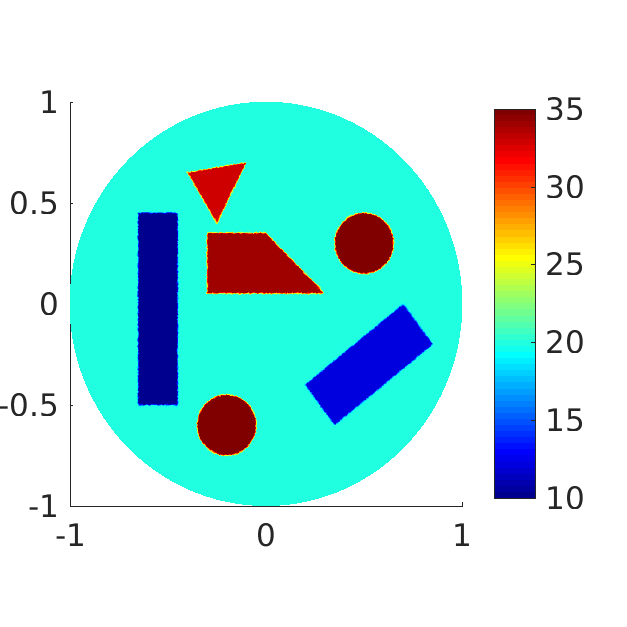}
   	\caption{Background truth of $\mu$}
   	\label{fig:discrete_true_mu}
\end{subfigure}
\begin{subfigure}[p]{.45\textwidth}
    \centering
    \includegraphics[scale=.4,trim={0cm 3cm 0cm 3cm},clip]{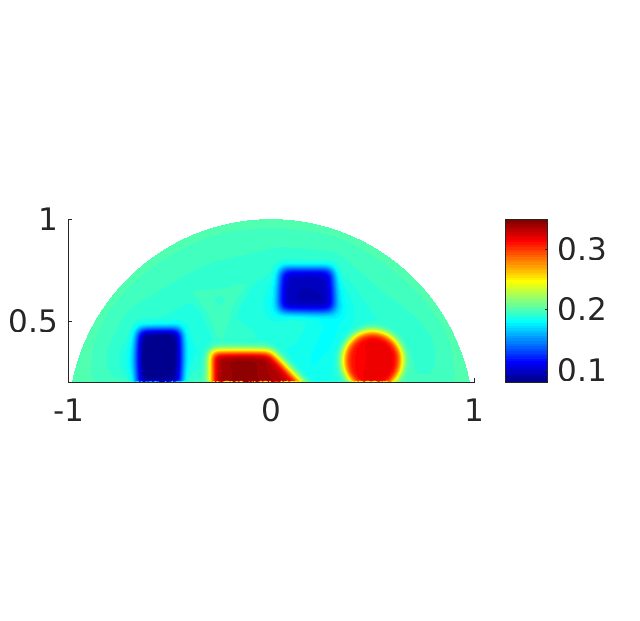}
   	\caption{Reconstruction of $D$}
   	\label{fig:discrete_recon_D}
\end{subfigure}
\begin{subfigure}[p]{.45\textwidth}
    \centering
    \includegraphics[scale=.4,trim={0cm 3cm 0cm 3cm},clip]{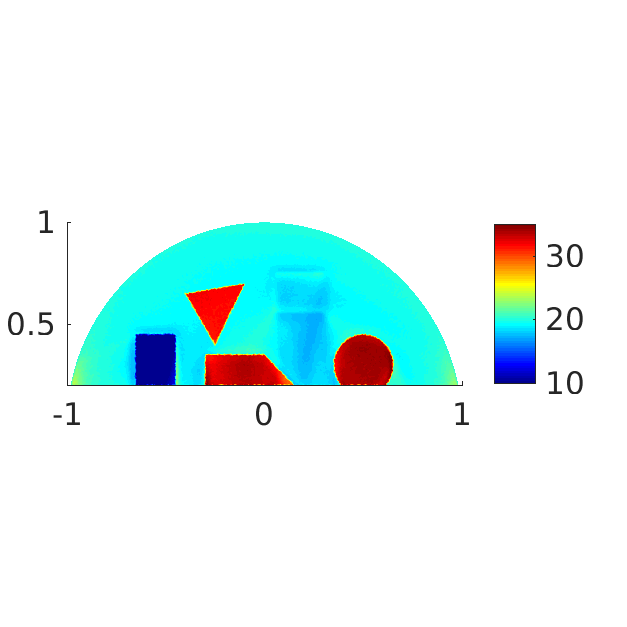}
   	\caption{Reconstruction of $\mu$}
   	\label{fig:discrete_recon_mu}
\end{subfigure}
\caption{Local reconstructions of discontinuous $D$ and $\mu$.}
\label{fig:discrete}
\end{figure}

\subsection{Blood vessel}

We try our algorithm on a more realistic example -- imaging the blood vessel of a piece of biological tissue.
We assign the tissue with proper diffusion and absorption values;
see Figure \ref{fig:real_true_D} and \ref{fig:real_true_mu}.
As demonstrated in Figure \ref{fig:real_recon_D} and \ref{fig:real_recon_mu}, all the features are well characterized by our results.
We only loose a little the contrast like the previous experiments.
The relative error is $4.74\%$ for $D$ and $2.47\%$ for $\mu$.

\begin{figure}[ptb]
\centering
\begin{subfigure}[p]{.45\textwidth}
    \centering
   	\includegraphics[scale=.4,trim={0cm 1cm 0cm 1cm},clip]{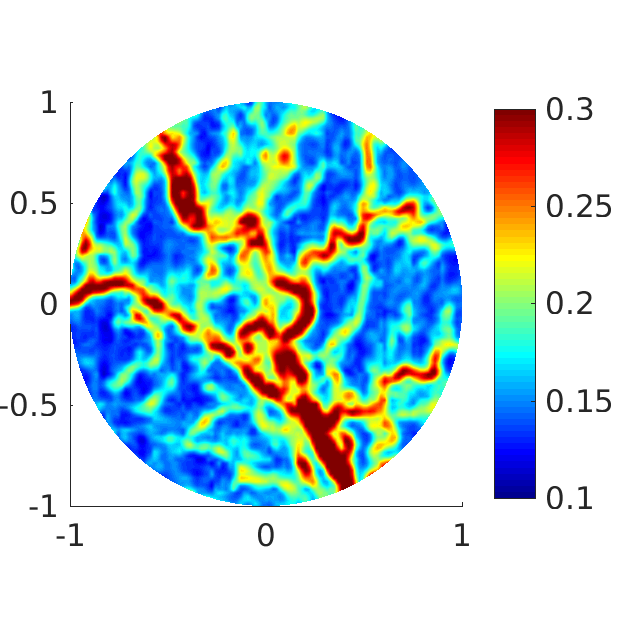}
   	\caption{Background truth of $D$}
   	\label{fig:real_true_D}
\end{subfigure}
\begin{subfigure}[p]{.45\textwidth}
    \centering
    \includegraphics[scale=.4,trim={0cm 1cm 0cm 1cm},clip]{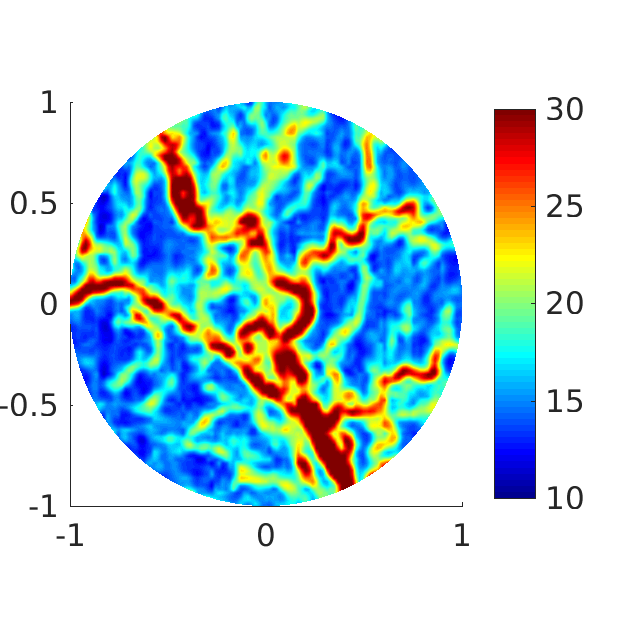}
   	\caption{Background truth of $\mu$}
   	\label{fig:real_true_mu}
\end{subfigure}
\begin{subfigure}[p]{.45\textwidth}
    \centering
    \includegraphics[scale=.4,trim={0cm 3cm 0cm 3cm},clip]{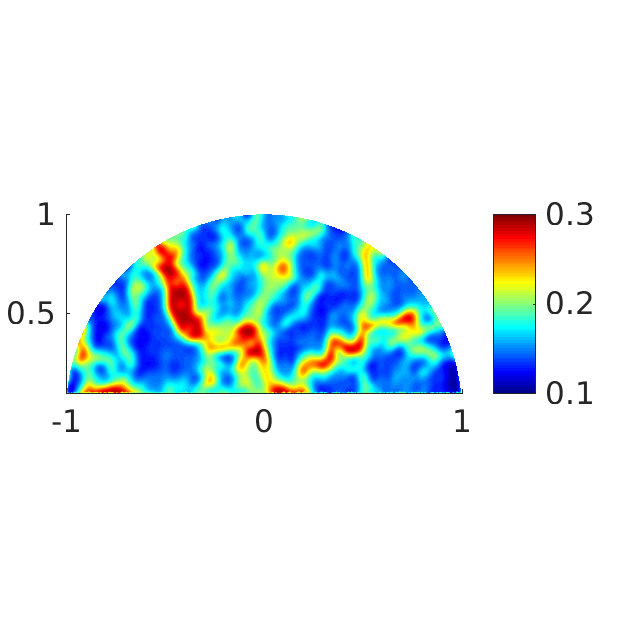}
   	\caption{Reconstruction of $D$}
   	\label{fig:real_recon_D}
\end{subfigure}
\begin{subfigure}[p]{.45\textwidth}
    \centering
    \includegraphics[scale=.4,trim={0cm 3cm 0cm 3cm},clip]{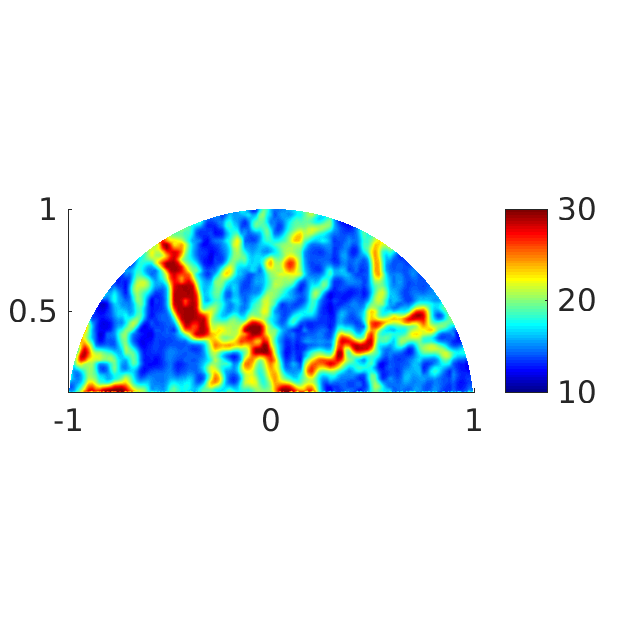}
   	\caption{Reconstruction of $\mu$}
   	\label{fig:real_recon_mu}
\end{subfigure}
\caption{Local reconstructions of $D$ and $\mu$ of the blood vessel.}
\label{fig:real}
\end{figure}

\section{Conclusion}

In this paper, we prove a H\"older stability of the quantitative PAT in a subregion
where the internal information is reliably provided 
based on the stability estimation of a Cauchy problem satisfied by the diffusion coefficient.
The exponent of the H\"older stability converges to a positive constant independent of the subregion as 
the subregion contracts towards the boundary.
\vskip .3cm
Numerical experiments demonstrates that it is possible to locally and efficiently reconstruct the diffusion and absorption coefficients
for smooth and even discontinuous media through the solution of an elliptic equation.

\bibliographystyle{plain}
\bibliography{refs}

\end{document}